\documentclass[12pt]{article}
\usepackage{amssymb}
\usepackage{amsthm}
\usepackage{amsmath}
\usepackage{graphicx}
\usepackage{enumerate}
\usepackage{pict2e}
\usepackage{framed}

\DeclareMathOperator{\Max}{Max}
\DeclareMathOperator{\Min}{Min}

\newtheorem{theorem}{Theorem}[section]
\newtheorem{definition}[theorem]{Definition}
\newtheorem{lemma}[theorem]{Lemma}
\newtheorem{proposition}[theorem]{Proposition}

\newtheorem{example}[theorem]{Example}
\newtheorem{corollary}[theorem]{Corollary}
\title{Adjointness of generalized Sasaki operations in posets}
\author{Ivan~Chajda and Helmut~L\"anger}
\date{}

\begin{document}
	
\footnotetext{Support of the research of the first author by the Czech Science Foundation (GA\v CR), project 24-14386L, entitled ``Representation of algebraic semantics for substructural logics'', and by IGA, project P\v rF~2024~011, is gratefully acknowledged.}

\maketitle
	
\begin{abstract}
The Sasaki projection and its dual were introduced as a mapping from the lattice of closed subspaces of a Hilbert space onto one of its segments. In a previous paper the authors showed that the Sasaki operations induced by the Sasaki projection and its dual form an adjoint pair in every orthomodular lattice. Later on the authors described large classes of algebras in which Sasaki operations can be defined and form an adjoint pair. The aim of the present paper is to extend these investigations to bounded posets with a unary operation. We introduce the so-called generalized Sasaki projection and its dual as well as the so-called generalized Sasaki operations induced by them. When treating these projections and operations we consider only so-called saturated posets, i.e.\ posets having the property that above any lower bound of two elements there is at least one maximal lower bound and below any upper bound of two elements there is at least one minimal upper bound. We prove that the generalized Sasaki operations are well-defined if and only if the poset in question is orthogonal. We characterize adjointness of the generalized Sasaki operations in different ways and show that adjointness is possible only if the unary operation is a complementation. Finally, we prove that in every saturated orthomodular poset the generalized Sasaki operations form an adjoint pair.
\end{abstract}
	
{\bf AMS Subject Classification:} 06A11, 06C15, 03G12, 81P10
	
{\bf Keywords:} Saturated poset, orthogonal poset, orthomodular poset, complementation, Sasaki projection, generalized Sasaki operations, adjoint pair

\section{Introduction}

It was proved by G.~Birkhoff and J.~von~Neumann \cite{BV} that the lattice $\mathbf L(\mathbf H)=\big(L(\mathbf H),\vee,\cap\big)$ of closed subspaces of a Hilbert space $\mathbf H$ is orthomodular. Independently, the same was proved also by K.~Husimi \cite H. Later on, U. Sasaki \cite S introduced the projection $p_a$ of $L(\mathbf H)$ onto its interval $[0,a]$ as follows:
\[
p_a(x):=(x\vee a')\wedge a.
\]
The name {\em Sasaki projection} was given later by M.~Nakamura \cite N.

The dual projection $\overline{p_a}$ from $L(\mathbf H)$ onto its interval $[a',1]$ given by
\[
\overline{p_a}(x):=\big(p_a(x')\big)'=(x\wedge a)\vee a'
\]
is called the {\em dual Sasaki projection}. We call binary operations defined by terms similar to those occurring in the definition of the (dual) Sasaki projection {\em Sasaki operations}, see \cite{CLa}. Using the (dual) Sasaki projection not necessarily in an orthomodular lattice, but in a lattice $(L,\vee,\wedge,{}')$ with a unary operation $\,'$ we define Sasaki operations $\odot$ and $\to$ as follows:
\begin{align*}
x\odot y & :=(x\vee y')\wedge y, \\
  x\to y & :=x'\vee(x\wedge y).
\end{align*}
We say that $\odot$ and $\to$ form an {\em adjoint pair} if for all $x,y,z\in L$
\[
x\odot y\le z\text{ if and only if }x\le y\to z.
\]
As shown by the authors in \cite{CL18}, \cite{CL23}, these operations form an adjoint pair in every orthomodular lattice. In \cite{CLa} the authors describe lattices with a unary operation for which $\odot$, $\to$ form an adjoint pair.

It was shown later in the logic of quantum mechanics that it cannot be formalized by orthomodular lattices since from physical reasons, the lattice join need not be defined everywhere. In fact, this join is defined for comparable elements and for elements $x$ and $y$ that are {\em orthogonal} to each other, i.e. $x\le y'$ or, equivalently, $y\le x'$, in symbols $x\perp y$. We obtain the following definition.

\begin{definition}\label{def1}
An {\em orthomodular poset} is a bounded poset $(P,\le,{}',0,1)$ with an antitone involution $\,'$ which is a complementation satisfying the following conditions for all $x,y\in P$:
\begin{enumerate}[{\rm(i)}]
\item If $x\le y$ then $x\vee y$ is defined,
\item if $x\le y$ then $y=x\vee(y'\vee x)'$ {\rm(}{\em orthomodularity}{\rm)}.
\end{enumerate}
An orthomodular poset that is also a lattice is called an {\em orthomodular lattice}.
\end{definition}

The aim of the present paper is to generalize the concept of Sasaki operations to bounded posets with a unary operation and determine when these generalized Sasaki operations form an adjoint pair. Because we work with posets where joins and meets need not exist, we must modify the definition of Sasaki operations as well as the concept of an adjoint pair. However, we show that there exists a rather broad class of such posets (which need not be necessarily orthomodular) having the mentioned property. For orthomodular posets, the analogy is very precise and the generalized Sasaki operations really form an adjoint pairs according to the new definition.

\section{Basic concepts}

Due to De Morgan's laws, (ii) of Definition~\ref{def1} can be equivalently rewritten in any of the following forms:
\begin{enumerate}[(i)]
\item If $x\le y$ then $y=x\vee(y\wedge x')$,
\item if $x\le y$ then $x=y\wedge(x\vee y')$.	
\end{enumerate}
There is the question if the concept of Sasaki operations can be generalized to bounded posets $(P,\le,{}',0,1)$ with a unary operation $\,'$ where the expressions $\Min U(a,b')\wedge b$ and $a'\vee\Max L(a,b)$ are defined for all $a,b\in P$.

In the following we often identify singletons with their unique element.

Let $\mathbf P=(P,\le)$ be a poset, $a,b\in P$ and $A,B\subseteq P$. Then $\Max A$ and $\Min A$ denote the set of all maximal and minimal elements of $A$, respectively. Observe that $\Max A$ and $\Min A$ are antichains. We define
\begin{align*}
   A & \le B\text{ if }a\le b\text{ for all }a\in A\text{ and all }b\in B, \\
   A & \le_1B\text{ if for every }a\in A\text{ there exists some }b\in B\text{ with }a\le b, \\
   A & \le_2B\text{ if for every }b\in B\text{ there exists some }a\in A\text{ with }a\le b, \\
L(A) & :=\{x\in P\mid x\le A\}, \\
U(A) & :=\{x\in P\mid A\le x\}.
\end{align*}
Instead of $L(\{a\})$, $L(\{a,b\})$, $L(A\cup\{b\})$, $L(A\cup B)$ and $L\big(U(A)\big)$ we simply write $L(a)$, $L(a,b)$, $L(A,b)$, $L(A,B)$ and $LU(A)$, respectively. Similarly, we proceed in analogous cases. It is evident that if $a\vee b$ is defined in $\mathbf P$ then $\Min U(a,b)=a\vee b$, and if $a\wedge b$ is defined in $\mathbf P$ then $\Max L(a,b)=a\wedge b$. We call $\mathbf P$ {\em saturated} if above any lower bound of two elements of $P$ there is at least one maximal lower bound and below any upper bound of two elements of $P$ there is at least one minimal upper bound. Notice that every finite poset and any lattice is saturated. This condition is weaker than the Ascending Chain Condition. For example, the chain of real numbers $(\mathbb R,\le)$ is saturated but does not satisfy the Ascending Chain Condition. Let $\,'$ be a unary operation on $P$. We call $(P,\le,{}')$ {\em orthogonal} if for all $a,b\in P$ the following holds:
\begin{enumerate}[(i)]
\item If $a\le b$ then $a\vee b'$ is defined,
\item if $a'\le b$ then $a\wedge b$ is defined.
\end{enumerate}
Observe that every orthomodular poset is orthogonal.

Let $(P,\le,{}',0,1)$ be a bounded orthogonal saturated poset and $a\in P$. We define the so-called {\em generalized Sasaki projection} and its dual as the following operators:
\begin{align*}
           p_a(x) & :=\Min U(x,a')\wedge a, \\
\overline{p_a}(x) & :=a'\vee\Max L(a,x)
\end{align*}
for all $x\in P$. We extend these operators from $P$ to $2^P$ by defining
\begin{align*}
           p_a(A) & :=\bigcup_{x\in A}p_a(x), \\
\overline{p_a}(A) & :=\bigcup_{x\in A}\overline{p_a}(x)
\end{align*}
for all $A\subseteq P$.

\begin{theorem}
Let $\mathbf P=(P,\vee,\wedge,{}',0,1)$ be a bounded orthogonal saturated poset, $a\in P$ and $A,B\subseteq P$ with $A\le_2B$. Then
\begin{enumerate}[{\rm(i)}]
\item $p_a(A)\subseteq[0,a]$,
\item $p_a(1)=a$,
\item $p_a(A)\le_2p_a(B)$.
\end{enumerate}
If, moreover, $\mathbf P$ is an orthomodular poset then
\begin{enumerate}
\item[{\rm(iv)}] $p_a(a')=0$ and $p_a(x)=x$ for all $x\in[0,a]$,
\item[{\rm(v)}] $p_a\big(p_a(A)\big)=p_a(A)$.
\end{enumerate}
\end{theorem}

\begin{proof}
\
\begin{enumerate}[(i)]
\item This is clear.
\item We have $p_a(1)=\Min U(1,a')\wedge a=1\wedge a=a$.
\item Let $b\in p_a(B)$. Then there exists some $c\in B$ with $b\in p_a(c)$. This means that there exists some $d\in\Min U(c,a')$ with $d\wedge a=b$. Because of $A\le_2B$ there exists some $e\in A$ with $e\le c$. Now $d\in \Min U(c,a')\subseteq U(c,a')\subseteq U(e,a')$. Since $\mathbf P$ is saturated there exists some $f\in\Min U(e,a')$ with $f\le d$. Because of $a'\le f$ the meet $f\wedge a$ is defined. Moreover, $f\wedge a\in p_a(e)\subseteq p_a(A)$. Together with $f\wedge a\le d\wedge a=b$ this shows $p_a(A)\le_2p_a(B)$.
\item Since $\mathbf P$ is an orthomodular poset we have
\begin{align*}
p_a(a') & =\Min U(a',a')\wedge a=a'\wedge a=0, \\
 p_a(x) & =\Min U(x,a')\wedge a=(x\vee a')\wedge a=x\text{ for all }x\in[0,a].
\end{align*} 
\item Due to (i) we have $p_a(A)\subseteq[0,a]$. Because of (iv) this implies $p_a(x)=x$ for all $x\in p_a(A)$. Altogether, we obtain
\[
p_a\big(p_a(A)\big)=\bigcup_{x\in p_a(A)}p_a(x)=\bigcup_{x\in p_a(A)}\{x\}=p_a(A).
\]
\end{enumerate}
\end{proof}

\section{Generalized Sasaki operations}

Using the generalized Sasaki projection and its dual we define our main concept as follows:

\begin{definition}\label{def2}
Let $(P,\le,{}')$ be a poset with a unary operation $\,'$ having the property that for all $x,y\in P$ the expressions $\Min U(x,y')\wedge y$ and $x'\vee\Max L(x,y)$ are defined. We define the so-called {\em generalized Sasaki operations} $\odot$ and $\to$ as the following operators:
\begin{align*}
x\odot y & :=\Min U(x,y')\wedge y, \\
  x\to y & :=x'\vee\Max L(x,y)
\end{align*}
for all $x,y\in P$.
\end{definition}

\begin{proposition}
Let $\mathbf P=(P,\le,{}',0,1)$ be a bounded saturated poset with a unary operation $'$. Then the generalized Sasaki operations $x\odot y$ and $x\to y$ are defined for all $x,y\in P$ if and only if $\mathbf P$ is orthogonal.
\end{proposition}

\begin{proof}
Let $a,b\in P$. First assume the generalized Sasaki operations $x\odot y$ and $x\to y$ to be defined for all $x,y\in P$. If $a\le b$ then $b\to a=b'\vee\Max L(b,a)=b'\vee a=a\vee b'$ is defined, and if $a'\le b$ then $b\odot a=\Min U(b,a')\wedge a=b\wedge a=a\wedge b$ is defined. Hence $\mathbf P$ is orthogonal. Conversely, assume $\mathbf P$ to be orthogonal. Since $b'\le\Min U(a,b')$ we have that $b\wedge\Min U(a,b')=\Min U(a,b')\wedge b=a\odot b$ is defined, and since $\Max L(a,b)\le a$ we have that $\Max L(a,b)\vee a'=a'\vee\Max L(a,b)=a\to b$ is defined. This shows that the generalized Sasaki operations $x\odot y$ and $x\to y$ are defined for all $x,y\in P$.
\end{proof}

Hence it makes sense to investigate the generalized Sasaki operations only for bounded orthogonal saturated posets.

\begin{example}\label{ex1}
Let $\mathbf P=(P,\le,{}',0,1)$ denote the complemented poset depicted in Figure~1:

\vspace*{-3mm}

\begin{center}
\setlength{\unitlength}{7mm}
\begin{picture}(7,8)
\put(2,1){\circle*{.3}}
\put(1,3){\circle*{.3}}
\put(3,3){\circle*{.3}}
\put(6,4){\circle*{.3}}
\put(1,5){\circle*{.3}}
\put(3,5){\circle*{.3}}
\put(2,7){\circle*{.3}}
\put(2,1){\line(-1,2)1}
\put(2,1){\line(1,2)1}
\put(2,1){\line(4,3)4}
\put(2,7){\line(-1,-2)1}
\put(2,7){\line(1,-2)1}
\put(2,7){\line(4,-3)4}
\put(1,3){\line(0,1)2}
\put(1,3){\line(1,1)2}
\put(3,3){\line(-1,1)2}
\put(3,3){\line(0,1)2}
\put(1.85,.3){$0$}
\put(.35,2.85){$a$}
\put(.35,4.85){$c$}
\put(3.4,2.85){$b$}
\put(3.4,4.85){$d$}
\put(1.85,7.4){$1$}
\put(6.4,3.85){$e$}
\put(2.7,-.75){{\rm Fig.~1}}
\put(.4,-1.75){{\rm Complemented poset $\mathbf P$}}
\end{picture}
\end{center}

\vspace*{10mm}

where the complementation $\,'$ is defined by the following table:
\[
\begin{array}{l|ccccccc}
x  & 0 & a & b & c & d & e & 1 \\
\hline
x' & 1 & e & e & e & e & c & 0
\end{array}
\]
That $\mathbf P$ is orthogonal can be seen as follows. Assume $x,y\in P$ and $x\le y$ and $x\vee y'$ not to be defined. Then $x,y'\in\{a,b\}$ which is impossible since $P'=\{0,c,e,1\}$. Now assume $x,y\in P$ and $x'\le y$ and $x\wedge y$ is not to be defined. Then $x,y\in\{c,d\}$ and hence $x'=e$ contradicting $x'\le y$. The operation tables of the generalized Sasaki operations look as follows:
\[
\begin{array}{c|ccccccc}
\odot & 0 & a & b & c & d & e & 1 \\
\hline
  0   & 0 & 0 & 0 & 0 & 0 & 0 & 0 \\
  a   & 0 & a & b & c & d & 0 & a \\
  b   & 0 & a & b & c & d & 0 & b \\
  c   & 0 & a & b & c & d & 0 & c \\
  d   & 0 & a & b & c & d & e & d \\
  e   & 0 & 0 & 0 & 0 & 0 & e & e \\
  1   & 0 & a & b & c & d & e & 1
\end{array}
\quad\quad\quad
\begin{array}{c|ccccccc}
\to & 0 & a & b & c & d & e & 1 \\
\hline
 0  & 1 & 1 & 1 & 1 & 1 & 1 & 1 \\
 a  & e & 1 & e & 1 & 1 & e & 1 \\
 b  & e & e & 1 & 1 & 1 & e & 1 \\
 c  & e & 1 & 1 & 1 & 1 & e & 1 \\
 d  & e & 1 & 1 & 1 & 1 & e & 1 \\
 e  & c & c & c & c & c & 1 & 1 \\
 1  & 0 & a & b & c & d & e & 1
\end{array}
\] 
\end{example}

The next lemma shows some basic properties of generalized Sasaki operations.

\begin{lemma}\label{lem1}
Let $\mathbf P=(P,\le,{}',0,1)$ be a bounded poset with a unary operation $'$. Then $\mathbf P$ satisfies the following identities:
\begin{enumerate}[{\rm(i)}]
\item $1\odot x\approx x$,
\item $0\to x\approx0'$,
\item $x\to0\approx x'$.
\end{enumerate}
If, moreover, $a\in P$ and $a\odot0$ is defined then $a\odot0=0$.
\end{lemma}

\begin{proof}
\
\begin{enumerate}[(i)]
\item $1\odot x\approx\Min U(1,x')\wedge x\approx1\wedge x\approx x$,
\item $0\to x\approx0'\vee\Max L(0,x)\approx0'\vee0\approx0'$,
\item $x\to0\approx x'\vee\Max L(x,0)\approx x'\vee0\approx x'$.
\end{enumerate}
If, moreover, $a\in P$ and $a\odot0$ is defined then $a\odot0\approx\Min U(a.0')\wedge0\approx0$.
\end{proof}

\section{Adjoint pairs of generalized Sasaki operations}

Let $(P,\le,{}',0,1)$ be a bounded orthogonal saturated poset. We are interested in the case when the generalized Sasaki operations form an adjoint pair. However, contrary to the case of orthomodular lattices mentioned in the introduction, now $x\odot y$ and $x\to y$ may be subsets of $P$ and hence we must modify adjointness as follows: We say that the generalized Sasaki operations form an {\em adjoint pair} if both (A1) and (A2) are satisfied:
\begin{enumerate}
\item[(A1)] $x,y,z\in P$ and $x\odot y\le_2z$ imply $x\le_1y\to z$,
\item[(A2)] $x,y,z\in P$ and $x\le_1y\to z$ imply $x\odot y\le_2z$.
\end{enumerate}

It is elementary to prove the following result.

\begin{lemma}
Let $(P,\le,{}',0,1)$ be a bounded orthogonal saturated poset and let $x,y\in P$. Then the following holds:
\begin{enumerate}[{\rm(i)}]
\item {\rm(A1)} implies $x\vee x'\approx1$,
\item {\rm(A2)} implies $x\wedge x'\approx0$,
\item if the generalized Sasaki operations form an adjoint pair then $\,'$ is a complementation,
\item {\rm(A2)} implies that $x\to y=1$ if and only if $x\le y$.
\end{enumerate}	
\end{lemma}

\begin{proof}
\
\begin{enumerate}[(i)]
\item Because of $1\odot x=x\le_21$ we have $1\le_1x\to1=x'\vee\Max L(x,1)=x'\vee x$.
\item Because of $0\le_1x\to0$ we have $x'\wedge x=\Min U(0,x')\wedge x=0\odot x\le_20$.
\item This follows from (i) and (ii).
\item If $x\to y=1$ then $1\le_1x\to y$ and hence according to Lemma~\ref{lem1} and (A2) we have $x=1\odot x\le_2y$, that means, $x\le y$. If, conversely, $x\le y$ then $x\to y=x'\vee\Max L(x,y)=x'\vee x=1$.
\end{enumerate}
\end{proof}

Hence, studying adjointness of generalized Sasaki operations in bounded orthogonal saturated posets $(P,\le,{}',0,1)$ makes sense only if $\,'$ is a complementation.

Now we can state the following theorem.

\begin{theorem}\label{th1}
Let $\mathbf P=(P,\le,{}',0,1)$ be a bounded orthogonal saturated poset. Then {\rm(A1)} is equivalent to any of the following conditions:
\begin{enumerate}[{\rm(i)}]
\item $\Min U(x,y')\approx y'\vee\big(\Min U(x,y')\wedge y\big)$,
\item $\Min U(x,y')\le_2y'\vee\big(\Min U(x,y')\wedge y\big)$ for all $x,y\in P$,
\item $x,y\in P$ and $x'\le y$ imply $y=x'\vee(y\wedge x)$.
\end{enumerate}
Moreover, {\rm(A2)} is equivalent to any of the following conditions:
\begin{enumerate}
\item[{\rm(iv)}] $x\wedge\big(\Max L(x,y)\vee x'\big)\approx \Max L(x,y)$,
\item[{\rm(v)}] $x\wedge\big(\Max L(x,y)\vee x'\big)\le_1\Max L(x,y)$ for all $x,y\in P$,
\item[{\rm(vi)}] $x,y\in P$ and $x\le y$ imply $x=(y'\vee x)\wedge y$.
\end{enumerate}
\end{theorem}

\begin{proof}
Let $a,b,c\in P$. \\
(i) $\Rightarrow$ (ii): This is clear. \\
(ii) $\Rightarrow$ (A1): \\
Suppose $a\odot b\le_2c$. Then there exists some $d\in\Min U(a,b')\wedge b$ with $d\le c$. Now $d\in L(b,c)$. Since $\mathbf P$ is saturated there exists some $e\in\Max L(b,c)$ with $d\le e$. Because of orthogonality of $\mathbf P$ the joins $b'\vee d$ and $b'\vee e$ are defined. Now $b'\vee d\in b'\vee\big(\Min U(a,b')\wedge b\big)$. From (ii) we conclude that there exists some $f\in\Min U(a,b')$ with $f\le b'\vee d$. Altogether we obtain
\[
a\le f\le b'\vee d\le b'\vee e\in b'\vee\Max L(b,c)=b\to c
\]
and hence $a\le_1b\to c$. \\
(A1) $\Rightarrow$ (iii): \\
Suppose $a'\le b$. Then $b\odot a=\Min U(b,a')\wedge a=b\wedge a$ and hence $b\odot a\le_2b\wedge a$. According to (A1) we obtain
\[
b\le_1a\to(b\wedge a)=a'\vee\Max L(a,b\wedge a)=a'\vee(b\wedge a)
\]
and hence $b\le a'\vee(b\wedge a)$ which is equivalent to $b=a'\vee(b\wedge a)$. \\
(iii) $\Rightarrow$ (i): \\
This follows from $y'\le z$ for all $z\in\Min U(x,y')$. \\
(iv) $\Rightarrow$ (v): \\
This is clear. \\
(v) $\Rightarrow$ (A2): \\
Suppose $a\le_1b\to c$. Then there exists some $d\in b'\vee\Max L(b,c)$ with $a\le d$. Now $d\in U(a,b')$. Since $\mathbf P$ is saturated there exists some $e\in\Min U(a,b')$ with $e\le d$. Because of orthogonality of $\mathbf P$ the meets $b\wedge d$ and $b\wedge e$ are defined. Now $b\wedge d\in b\wedge\big(\Max L(b,c)\vee b'\big)$. From (v) we conclude that there exists some $f\in\Max L(b,c)$ with $b\wedge d\le f$. Altogether we obtain
\[
a\odot b=\Min U(a,b')\wedge b\ni e\wedge b\le b\wedge d\le f\le c
\]
and hence $a\odot b\le_2c$. \\
(A2) $\Rightarrow$ (vi): \\
Suppose $a\le b$. Then $b\to a=b'\vee\Max L(b,a)=b'\vee a$ and hence $b'\vee a\le_1b\to a$. According to (A2) we obtain
\[
(b'\vee a)\wedge b=\Min U(b'\vee a,b')\wedge b=(b'\vee a)\odot b\le_2a
\]
and hence $(b'\vee a)\wedge b\le a$ which is equivalent to $(b'\vee a)\wedge b=a$. \\
(vi) $\Rightarrow$ (iv): \\
This follows from $z\le x$ for all $z\in\Max L(x,y)$.
\end{proof}

\begin{corollary}\label{cor1}
The generalized Sasaki operations form an adjoint pair if and only if $\mathbf P$ satisfies the identities {\rm(i)} and {\rm(iv)} which is further equivalent to the fact that conditions {\rm(iii)} and {\rm(vi)} hold. Since every orthomodular poset is a bounded orthogonal poset satisfying {\rm(iii)} and {\rm(vi)} we see that in a saturated orthomodular poset the generalized Sasaki operations form an adjoint pair. Observe that {\rm(iii)} and {\rm(vi)} contain only two variables.
\end{corollary}

However, the class of saturated orthomodular posets is not the only important class of complemented posets where generalized Sasaki operations form an adjoint pair.

Recall that a poset $(P,\le)$ is called {\em modular} if for all $x,y,z\in P$ with $x\le z$ we have $L\big(U(x,y),z\big)=LU\big(x,L(y,z)\big)$ or, equivalently, $U\big(x,L(y,z)\big)=UL\big(U(x,y),z\big)$.

\begin{corollary}
Let $(P,\le,{}',0,1)$ be a complemented orthogonal modular poset. Then it satisfies {\rm(iii)} and {\rm(vi)} of Theorem~\ref{th1}. Hence, in every saturated complemented orthogonal modular poset the generalized Sasaki operations form an adjoint pair.
\end{corollary}

\begin{proof}
Let $a,b\in P$. If $a'\le b$ then
\[
U(b)=UL(b)=UL(1,b)=UL\big(U(a',a),b\big)=U\big(a',L(a,b)\big)=U(a',a\wedge b)=U\big(a'\vee(b\wedge a)\big)
\]
and hence $b=a'\vee(b\wedge a)$. If $a\le b$ then
\[
L(a)=LU(a)=LU(a,0)=LU\big(a,L(b',b)\big)=L\big(U(a,b'),b\big)=L(a\vee b',b)=L\big((b'\vee a)\wedge b\big)
\]
and hence $a=(b'\vee a)\wedge b$.	
\end{proof}

It is worth noticing that every distributive poset is modular and hence also in every saturated orthogonal Boolean poset the generalized Sasaki operations form an adjoint pair.

\begin{example}
In the poset of Example~\ref{ex1} the generalized Sasaki operations do not form an adjoint pair since $1\le_11=c\to a$, but $1\odot c=c\not\le_2a$.
\end{example}

\begin{example}
Let $\mathbf M_3=(M_3,\vee,\wedge)$ denote the lattice visualized in Figure~2:

\vspace*{-3mm}

\begin{center}
\setlength{\unitlength}{7mm}
\begin{picture}(6,6)
\put(3,1){\circle*{.3}}
\put(1,3){\circle*{.3}}
\put(3,3){\circle*{.3}}
\put(5,3){\circle*{.3}}
\put(3,5){\circle*{.3}}
\put(3,1){\line(-1,1)2}
\put(3,1){\line(0,1)4}
\put(3,1){\line(1,1)2}
\put(3,5){\line(-1,-1)2}
\put(3,5){\line(1,-1)2}
\put(2.85,.3){$0$}
\put(.35,2.85){$a$}
\put(3.4,2.85){$b$}
\put(5.4,2.85){$c$}
\put(2.85,5.4){$1$}
\put(2.2,-.75){{\rm Fig.~2}}
\put(.4,-1.75){{\rm Modular lattice $\mathbf M_3$}}
\end{picture}
\end{center}

\vspace*{10mm}

Define a unary operation $\,'$ on $M_3$ by the following table:
\[
\begin{array}{l|ccccc}
x  & 0 & a & b & c & 1 \\
\hline
x' & 0 & b & c & a & 1
\end{array}
\]
Then $\,'$ is a complementation on $\mathbf M_3$ and $\mathbf P:=(M_3,\le,{}',0,1)$ is an orthogonal saturated poset that is not an orthomodular poset since $\,'$ is not an involution. Let $x,y\in M_3$. First assume $x'\le y$. \\
If $x'=0$ then $x'\vee(y\wedge x)=0\vee(y\wedge1)=y$, \\
if $y=1$ then $x'\vee(y\wedge x)=x'\vee(1\wedge x)=x'\vee x=1=y$, \\
if $x'=y$ then $x'\vee(y\wedge x)=y\vee(y\vee x)=y$. \\
Hence $\mathbf P$ satisfies {\rm(iii)} of Theorem~\ref{th1}. Now assume $x\le y$. \\
If $x=0$ then $(y'\vee x)\wedge y=(y'\vee0)\wedge y=y'\wedge y=0=x$, \\
if $y=1$ then $(y'\vee x)\wedge y=(0\vee x)\wedge1=x$, \\
if $x=y$ then $(y'\vee x)\wedge y=(x'\vee x)\wedge x=x$. \\
This shows that $\mathbf P$ satisfies {\rm(vi)} of Theorem~\ref{th1}. According to Corollary~\ref{cor1} the generalized Sasaki operations form an adjoint pair.
\end{example}

\begin{example}
Let $\mathbf L$ denote the orthomodular lattice depicted in Figure~3:

\vspace*{-3mm}

\begin{center}
\setlength{\unitlength}{7mm}
\begin{picture}(12,8)
\put(6,1){\circle*{.3}}
\put(1,3){\circle*{.3}}
\put(3,3){\circle*{.3}}
\put(5,3){\circle*{.3}}
\put(7,3){\circle*{.3}}
\put(9,3){\circle*{.3}}
\put(11,3){\circle*{.3}}
\put(1,5){\circle*{.3}}
\put(3,5){\circle*{.3}}
\put(5,5){\circle*{.3}}
\put(7,5){\circle*{.3}}
\put(9,5){\circle*{.3}}
\put(11,5){\circle*{.3}}
\put(6,7){\circle*{.3}}
\put(6,1){\line(-5,2)5}
\put(6,1){\line(-3,2)3}
\put(6,1){\line(-1,2)1}
\put(6,1){\line(1,2)1}
\put(6,1){\line(3,2)3}
\put(6,1){\line(5,2)5}
\put(6,7){\line(-5,-2)5}
\put(6,7){\line(-3,-2)3}
\put(6,7){\line(-1,-2)1}
\put(6,7){\line(1,-2)1}
\put(6,7){\line(3,-2)3}
\put(6,7){\line(5,-2)5}
\put(1,3){\line(0,1)2}
\put(1,3){\line(1,1)2}
\put(3,3){\line(-1,1)2}
\put(3,3){\line(1,1)2}
\put(5,3){\line(-1,1)2}
\put(5,3){\line(0,1)2}
\put(7,3){\line(0,1)2}
\put(7,3){\line(1,1)2}
\put(9,3){\line(-1,1)2}
\put(9,3){\line(1,1)2}
\put(11,3){\line(-1,1)2}
\put(11,3){\line(0,1)2}
\put(5.85,.3){$0$}
\put(.35,2.85){$a$}
\put(2.35,2.85){$b$}
\put(4.35,2.85){$c$}
\put(7.4,2.85){$d$}
\put(9.4,2.85){$e$}
\put(11.4,2.85){$f$}
\put(.35,4.85){$c'$}
\put(2.35,4.85){$b'$}
\put(4.35,4.85){$a'$}
\put(7.4,4.85){$f'$}
\put(9.4,4.85){$e'$}
\put(11.4,4.85){$d'$}
\put(5.85,7.4){$1$}
\put(5.2,-.75){{\rm Fig.~3}}
\put(3,-1.75){{\rm Orthomodular lattice $\mathbf L$}}
\end{picture}
\end{center}

\vspace*{10mm}

According to Corollary~\ref{cor1} the generalized Sasaki operations on $\mathbf L$ form an adjoint pair and $\mathbf L$ contains a sublattice of the form $\mathbf O_6$ visualized in Figure~4:

\vspace*{-3mm}

\begin{center}
\setlength{\unitlength}{7mm}
\begin{picture}(4,8)
\put(2,1){\circle*{.3}}
\put(1,3){\circle*{.3}}
\put(3,3){\circle*{.3}}
\put(1,5){\circle*{.3}}
\put(3,5){\circle*{.3}}
\put(2,7){\circle*{.3}}
\put(2,1){\line(-1,2)1}
\put(2,1){\line(1,2)1}
\put(1,3){\line(0,1)2}
\put(3,3){\line(0,1)2}
\put(2,7){\line(-1,-2)1}
\put(2,7){\line(1,-2)1}
\put(1.85,.3){$0$}
\put(.35,2.85){$c$}
\put(.35,4.85){$a'$}
\put(3.4,2.85){$d$}
\put(3.4,4.85){$f'$}
\put(1.85,7.4){$1$}
\put(1.2,-.75){{\rm Fig.~4}}
\put(.3,-1.75){{\rm Sublattice of $\mathbf L$}}
\end{picture}
\end{center}

\vspace*{10mm}

Observe that this sublattice is not a subalgebra of $\mathbf L$ since it is not closed under the complementation. On the other hand, the following holds: If a complemented lattice contains a subalgebra of the form $\mathbf O_6$ depicted in Figure~5:

\vspace*{-3mm}

\begin{center}
\setlength{\unitlength}{7mm}
\begin{picture}(4,8)
\put(2,1){\circle*{.3}}
\put(1,3){\circle*{.3}}
\put(3,3){\circle*{.3}}
\put(1,5){\circle*{.3}}
\put(3,5){\circle*{.3}}
\put(2,7){\circle*{.3}}
\put(2,1){\line(-1,2)1}
\put(2,1){\line(1,2)1}
\put(1,3){\line(0,1)2}
\put(3,3){\line(0,1)2}
\put(2,7){\line(-1,-2)1}
\put(2,7){\line(1,-2)1}
\put(1.85,.3){$0$}
\put(.35,2.85){$x$}
\put(.35,4.85){$z$}
\put(3.4,2.85){$y$}
\put(3.4,4.85){$u$}
\put(1.85,7.4){$1$}
\put(1.2,-.75){{\rm Fig.~5}}
\put(-2.4,-1.75){{\rm Subalgebra of a complemented lattice}}
\end{picture}
\end{center}

\vspace*{10mm}

then the generalized Sasaki operations do not form an adjoint pair. On the contrary, assume that they form an adjoint pair. Then because of $z'\in\{y,u\}$ we have
\[
1\le_11=z'\vee x=z'\vee\Max L(z,x)=z\to x
\]
and hence
\[
z=1\wedge z=\Min U(1,z')\wedge z=1\odot z\le_2x,
\]
a contradiction.
\end{example}








Authors' addresses:

Ivan Chajda \\
Palack\'y University Olomouc \\
Faculty of Science \\
Department of Algebra and Geometry \\
17.\ listopadu 12 \\
771 46 Olomouc \\
Czech Republic \\
ivan.chajda@upol.cz

Helmut L\"anger \\
TU Wien \\
Faculty of Mathematics and Geoinformation \\
Institute of Discrete Mathematics and Geometry \\
Wiedner Hauptstra\ss e 8-10 \\
1040 Vienna \\
Austria, and \\
Palack\'y University Olomouc \\
Faculty of Science \\
Department of Algebra and Geometry \\
17.\ listopadu 12 \\
771 46 Olomouc \\
Czech Republic \\
helmut.laenger@tuwien.ac.at

\begin{thebibliography}9
\bibitem{BV}
G.~Birkhoff and J.~von~Neumann, The logic of quantum mechanics. Ann.\ of Math.\ {\bf37} (1936), 823--843.
\bibitem{CL18}
I.~Chajda and H.~L\"anger, Weakly orthomodular and dually weakly orthomodular posets. Asian-Eur.\ J.\ Math.\ {\bf11} (2018), 1850093 (18pp.).
\bibitem{CL23}
I.~Chajda and H.~L\"anger, Operator residuation in orthomodular posets of finite height. Fuzzy Sets and Systems {\bf467} (2023), 108589 (11 pp.).
\bibitem{CLa}
I.~Chajda and H.~L\"anger, Algebras and varieties where Sasaki operations form an adjoint pair. Miskolc Math.\ Notes (submitted).
\bibitem H
K.~Husimi, Studies on the foundation of quantum mechanics. I. Proc.\ Phys.-Math.\ Soc.\ Japan {\bf19} (1937), 766--789.
\bibitem N
M.~Nakamura, The permutability in a certain orthocomplemented lattice. Kodai Math.\ Sem.\ Rep.\ {\bf9} (1957), 158--160.
\bibitem S
U.~Sasaki, Orthocomplemented lattices satisfying the exchange axiom. J.\ Sci.\ Hiroshima Univ.\ Ser.\ A {\bf17} (1954), 293--302.
\end{thebibliography}
\end{document}